\documentclass[a4paper,11pt]{article} 

\usepackage{amsmath,amssymb,color} 
\usepackage{latexsym} 
\usepackage{graphicx}

\newtheorem{theorem}{Theorem} 
\newtheorem{lemma}[theorem]{Lemma} 
\newtheorem{proposition}[theorem]{Proposition}
 
\newtheorem{conjecture}[theorem]{Conjecture}

\newcommand{\qed}{~$\Box $} 

\date{\empty}

\begin{document} 

\title{Odd Edge Colorings of Graphs with Odd Order} 

\author{
 Mikio Kano$^1$
\footnote{mikio.kano.math@vc.ibaraki.ac.jp} 
 Shun-ichi Maezawa$^2$
\footnote{maezawa.mw@gmail.com} 
 Kenta Ozeki$^3$
\footnote{ozeki-kenta-xr@ynu.ac.jp} 
\and
$^{1}$ Ibaraki University, Hitachi, Ibaraki, Japan
\and
$^{2}$  Nihon University, Setagaya-ku, Tokyo, Japan 
\and
$^3$ Yokohama National University, Yokohama, Kanagawa, Japan
} 

\maketitle

\begin{abstract}
An {\em odd subgraph} of  a graph is a subgraph in which every vertex has odd degree.
A graph $G$ is said to be {\em odd $k$-edge-colorable} if there exists an edge-coloring $E(G) \rightarrow \{1,2, \ldots, k\}$ such that
each non-empty color class induces an odd subgraph of $G$.
The {\em odd chromatic index} of $G$, denoted by $\chi'_o(G)$, is the minimum $k$ for which $G$ is odd $k$-edge-colorable.
In this paper, 
we prove that every $4$-connected simple graph of odd order is odd 3-edge-colorable,
and show that the $4$-connectedness assumption is necessary.
We also prove that 
for a connected Eulerian graph $G$ of odd order,
there exists an edge $e$ such that $G-e$ is odd $2$-edge-colorable.

\end{abstract} 

\section{Introduction} 

In this paper, we deal with multigraphs and simple graphs, where a multigraph may have multiple edges but has no loops and a simple graph has neither loops nor multiple edges. However, we mainly consider simple graphs, and so we briefly call a simple graph a {\em graph}.
Let $G$ be a multigraph. For a vertex $v$ of $G$, let $N_G(v)$ and $\deg_G(v)$ denote the {\em neighborhood} and the {\em degree} of $v$ in $G$, respectively.
An {\em odd subgraph} of $G$ is a subgraph in which every vertex has odd degree, and
a spanning odd subgraph is called an {\em odd factor} of $G$.

We say that $G$ is {\em odd $k$-edge-colorable} if there exists an edge-coloring $E(G) \rightarrow \{1,2, \ldots, k\}$ such that each non-empty color class induces an odd subgraph of $G$.
The {\em odd chromatic index} of $G$, denoted by $\chi'_o(G)$, is the minimum $k$ for which $G$ is odd $k$-edge-colorable.
Odd edge-coloring has been studied for a long time. In 1991, Pyber proved that every tree is odd $2$-edge-colorable, and
this result together with a well-known fact that every connected multigraph of even order has an odd factor (Theorem~6.7 of \cite{AK-2011}) implies Theorem~\ref{even}.

\begin{theorem}[Pyber~\cite{Pyber-1991}]\label{even}
Every connected graph $G$ of even order satisfies $\chi'_o(G) \leq 3$.
\end{theorem}

By the same way, we can show that Theorem \ref{even} holds even for multigraphs.
Pyber also obtained a result on the odd chromatic index of a graph of odd order.
Later Petru\v{s}evski extended this to multigraph as follows.

\begin{theorem}[Petru\v{s}evski \cite{Petrusevski-2018}]\label{thm-P}
Every connected multigraph $G$ of order at least $4$ satisfies $\chi'_o(G) \leq 4$.
\end{theorem}

The upper bound of $\chi'_o(G)$ in Theorem~\ref{thm-P} is tight;
that is, there are infinitely many graphs that are not odd $3$-edge-colorable
(see below).
Motivated by this fact, many researchers have studied sufficient conditions under which a graph $G$ satisfies $\chi'_o(G) \le 3$.
Petru\v{s}evski \cite[Proposition 2.5]{Petrusevski-2018} proved the following theorem.

\begin{theorem}[Petru\v{s}evski \cite{Petrusevski-2018}]
Let $G$ be a connected multigraph of odd order.
If $G$ contains a bridge, then $\chi'_o(G) \leq 3$.
\label{thm-bridge}
\end{theorem}

Considering Theorem~\ref{thm-bridge}, it is natural to study the opposite side,
namely graphs with high connectivity, which are the main focus of this paper.

Pyber~\cite{Pyber-1991} pointed out that the wheel $W_4$,
the wheel with four spokes, satisfies $\chi'_o(W_4)=4$,
and asked whether there exist infinitely many connected graphs $G$ with $\chi'_o(G)=4$.
M\'{a}trai~\cite{Matrai-2006} answered this question affirmatively,
but all of the graphs he constructed contain a cut-vertex
and hence are not $2$-connected.
Later, Atanasov, Petru\v{s}evski, and \v{S}krekovski~\cite[Proposition~3.3]{APS-2016}
proved the following theorem,
which gives infinitely many $2$-connected graphs $G$ with $\chi'_o(G)=4$.
Note that a {\em cubic graph} is a 3-regular graph.

\begin{theorem}[Atanasov et al. \cite{APS-2016}]
Let $G$ be a multigraph obtained from a connected cubic bipartite multigraph by a single edge subdivision. Then $\chi'_o(G)=4$.
\label{thm-4}
\end{theorem}

However, graphs with higher connectivity have not been studied.
Considering this situation,
we prove the following theorem.

\begin{theorem}
Every $4$-connected graph $G$ of odd order satisfies $\chi'_o(G) \leq 3$.
\label{thm-4-conn}
\end{theorem}

It is natural to consider the case of multigraphs in Theorem~\ref{thm-4-conn},
but we do not know whether it holds.

We point out that the $4$-connectedness assumption is necessary 
because there are infinitely many $3$-connected graphs $G$ with $\chi'_o(G) = 4$.
We prove this fact in Section \ref{best_sec} by employing the same idea as 
in the proof of Theorem~\ref{thm-4} in \cite{APS-2016}.
Note that all graphs we have constructed need vertices of degree $3$,
and hence we pose the following conjecture.

\begin{conjecture}
Let $G$ be a $3$-connected graph of odd order.
If the minimum degree of $G$ is at least $4$, then $\chi'_o(G) \leq 3$.
\label{conj-3-conn}
\end{conjecture}

Kano, Katona, and Varga~\cite{KKV-2018} gave
a criterion for graphs $G$ to satisfy $\chi'_o(G) \le 2$,
and a polynomial time algorithm for decomposing a multigraph into two odd subgraphs
or showing the non-existence of such a decomposition.
Thus, it is natural to study a similar criterion for $\chi'_o(G)\le 3$.
By Theorem~\ref{thm-4-conn},
all graphs $G$ with $\chi'_o(G)=4$ are not $4$-connected,
suggesting that one possible approach is to focus on cut-sets of order at most $3$.

We prove Theorem \ref{thm-4-conn} by considering some cases.
Since some of them hold even under weaker connectivity assumptions
and would be helpful to approach Conjecture \ref{conj-3-conn},
we state them below separately.

\begin{theorem}
\label{thm-2}
Let $G$ be a $3$-connected graph of odd order.
If $G$ has at least two vertices of even degree,
then $\chi'_o(G)\le 3$.
\end{theorem}

\begin{theorem}
\label{thm-32}
Let $G$ be a connected graph of odd order.
Suppose that $G$ has exactly one vertex of even degree, say $w$.
If $w$ is adjacent to all the other vertices in $G$
and $G-w$ is connected,
then $\chi'_o(G)\le 3$,
unless $G$ is isomorphic to $W_4$.
\end{theorem}

Petru\v{s}evski \cite[Lemma 3.2]{Petrusevski-2018} proved 
that 
for a connected graph $G$ of odd order,
if there are two adjacent vertices $w$ and $u$ of even degree
such that both $G-w$ and $G -\{w,u\}$ are connected,
then $\chi'_o(G)\le 3$.
Theorem \ref{thm-2} extends this result assuming the $3$-connectedness.

Some other results on odd edge-coloring can be found 
in \cite{KKV-2018, LPS-2015, Petrusevski-2018, PS-2023}.
Petru\v{s}evski and \v{S}krekovski~\cite{PS-2021} conjectured that 
every connected multigraph $G$ with $\chi'_o(G)=4$ becomes odd $3$-edge-colorable by removing a particular edge.
This was recently proved by Lin, Petru\v{s}evski, and Yang~\cite{LY-2024}.
We show the following theorem,
which is an analog for the case of odd $2$-edge-colorable.

\begin{theorem}
Let $G$ be a connected Eulerian graph of  odd order.
Then there exists an edge $e$ such that $G-e$ is odd $2$-edge-colorable.
\label{Eulerian}
\end{theorem}

This article is organized as follows. 
In the next section,
we construct infinitely many $3$-connected graphs $G$ with $\chi'_o(G) = 4$,
which shows the best possibility of the assumption of Theorem \ref{thm-4-conn}.
We give some lemmas in Section \ref{lemmas_sec},
and then give proofs to our main theorems in Section \ref{proofs_sec}.
In the last section,
we give some observation toward Conjecture \ref{conj-3-conn}.

\section{$3$-connected graphs that are not odd $3$-edge-colorable}
\label{best_sec}

In this section,
we show that there are infinitely many $3$-connected graphs $G$ with $\chi'_o(G) = 4$.
In fact, 
inspired by Theorem \ref{thm-4} due to \cite{APS-2016},
we prove the following theorem.

\begin{theorem}
\label{hanrei}
Suppose that a connected graph $G$ has a vertex $w$ satisfying  the following three conditions:
\begin{itemize}
\item
$G-w$ is a bipartite graph.
\item
$w$ has degree $4$ 
and all the other vertices have degree $3$.
\item
$w$ is connected to each partite set of $G-w$ by exactly two edges.
\end{itemize}
Then $\chi'_o(G) =4$.
\end{theorem}

Note that the wheel $W_4$ satisfies the conditions in Theorem \ref{hanrei},
where $w$ corresponds to the center of $W_4$.
Thus, Theorem \ref{hanrei} shows which property of $W_4$ causes it to be non-odd 3-edge-colorable.

Notice also that there exist infinitely many $3$-connected graphs $G$ satisfying the conditions in Theorem \ref{hanrei}.
For example, from a $3$-connected cubic bipartite graph,
subdivide two appropriate edges and identify the obtained vertices.

\medskip \noindent
{\em Proof of Theorem \ref{hanrei}.}
 It suffices to show that $\chi'_o(G) \ge 4$ by Theorem~\ref{thm-P}.
Let $G' = G - w$, which is a bipartit graph,
and let $X, Y$ be its partite sets.
Suppose that there exists an odd $3$-edge-coloring of $G$ with the color set $\{1, 2, 3\}$.
Without loss of generality, we may assume that
one edge between $w$ and a vertex in $X$ are colored with $1$,
and the remaining three edges incident to $w$ are colored with $2$.

For $i \in \{1,2,3\}$,
let $X_i$ be the set of vertices $x$ in $X$
such that the edges incident to $x$ is colored entirely with $i$,
and let $X_0 = X - (X_1 \cup X_2 \cup X_3)$.
Note that each vertex in $X_0$ is incident to exactly one edge of each of the colors $1$, $2$ and $3$.
Analogously, we employ notation $Y_1$, $Y_2$, $Y_3$ and $Y_0$
for the respective subsets of $Y$.

By double counting the edges of color $1$ between $X$ and $Y$,
and those of color $2$,
we derive the following two equalities:
\begin{eqnarray*}
3|X_1| + |X_0| -1 &=& 3|Y_1| + |Y_0|, \\
3|X_2| + |X_0| -1 &=& 3|Y_2| + |Y_0| -2.
\end{eqnarray*}
By subtracting these two equalities,
we obtain $3(|X_1| - |X_2|) = 3(|Y_1| - |Y_2|) + 2$,
which is a contradiction.
\qed

\section{Lemmas} 
\label{lemmas_sec}

For two subsets $X$ and $Y$ of a set, we denote their symmetric difference by $X \triangle Y$.
If $X\subseteq Y$, then we often write $Y-X$ for $Y\setminus X$. 
Let $G$ be a graph and  $v$ be a vertex of $G$. 
Then the minimum degree of $G$ is
denoted by $\delta(G)$, and
the set of edges of $G$ incident with $v$ is denoted by $E_G(v)$.
We say that a graph $G$ is decomposed into subgraphs $H_1, H_2, \ldots, H_m$
if the edge set of $G$ satisfies $E(G)=E(H_1) \cup E(H_2) \cup \cdots \cup E(H_m)$,
and the sets $E(H_i)$ are pairwise disjoint.

For a cycle or a path $Q$ in a graph $G$,
an edge of $E(G)-E(Q)$ joining two vertices of $Q$ is called a \emph{chord} of $Q$.
If $G$ has no chord of $Q$, then $Q$ is called a {\em chordless cycle} or a {\em chordless path}.

\subsection{Lemmas for an odd subgraph}

We introduce three lemmas,
which have been shown in some literatures
and used widely to show some results on odd edge-coloring.

The first lemma can be easily shown by taking a minimal $T$-join,
see \cite{BM-2008} for $T$-join.
Thus, we omit its proof.

\begin{lemma}
\label{sp_odd_sub_lemma}
Let $G$ be a connected graph and let $S \subseteq V(G)$ with $|S| \equiv 0 \pmod{2}$.  
Then $G$ has a spanning subgraph $F$ such that 
every non-trivial component of $F$ contains a vertex of $S$
and for each vertex $v$ in $G$,
$\deg_F(v)$ is odd if $v \in S$, and otherwise $\deg_F(v) \in \{0, 2, 4,  \ldots\}$.
\end{lemma}

The following lemma is stated in \cite{Petrusevski-2018}, 
and its proof is based on the fact that every tree is odd $2$-edge-colorable.
The argument constructs a new forest by splitting the vertex $w$ of a graph $G$ into $\deg_G(w)$ new end-vertices,  and then joining each vertex of $N_G(w)$ to one of these new end-vertices.

\begin{lemma}[Proposition 2.3 in \cite{Petrusevski-2018}]
\label{mfv}
Let $G$ be a connected graph that has a vertex $w$
such that $G - w$ is a forest.
Then there exists a decomposition of $G$ into two subgraphs $H_1, H_2$
such that $\deg_{H_i} (x)$ is odd
for every $x \in V(H_i) - \{w\}$ and each $i \in \{1,2\}$.
\end{lemma} 

The following lemma also can be shown by using the above consturuction 
for Lemma~\ref{sp_odd_sub_lemma}.

\begin{lemma}[Proposition 2.6(i) in \cite{Petrusevski-2018}]
\label{all_edges}
Let $G$ be a connected graph of even order,
and let $w$ be a vertex of odd degree.
If $G - w$ is connected,
then $G$ has an odd factor $F$ such that 
$G - E(F)$ is a forest and $F$ contains all edges in $E_G(w)$.
\end{lemma}

\subsection{Lemmas for subgraphs whose removal results in a connected graph}

In our argument, it is important to find suitable subgraphs
whose removal results still in a connected graph.
To do that, we first introduce the two theorems which were shown by Tutte. 

\begin{theorem}[Tutte \cite{Tutte-1961}]
\label{Tutte_lemma}
For every $3$-connected graph $G$ with an edge $e$ and a vertex $v$ that is not incident to $e$,
there exists a chordless cycle $C$ in $G - v$ such that $e \in E(C)$ and $G - V (C)$ is connected. 
\end{theorem}

\begin{theorem}[Tutte \cite{Tutte},  \cite{CFS, KO-2011}]
Let $G$ be a $3$-connected graph,
and let $w,u$ be two nonadjacent vertices.
Then there exists a chordless path $P$ connecting $w$ and $u$ 
such that $G - V(P)$ is connected.
\label{Tutte}
\end{theorem}

Note that Tutte \cite{Tutte} originally proved the weaker version of Theorem \ref{Tutte}
without the condition ``chordless''.
A complete proof of the full statement can be found in several papers,
for example, see \cite{CFS, KO-2011},
where the condition ``chordless'' is not explicitly stated,
but their proofs essentially establish Theorem~\ref{Tutte}.

The following is an analog for $2$-connected Eulerian graph,
which will be used for the proof of Theorem \ref{Eulerian}.
It is a special case of a known theorem,
such as \cite[Problem 6(b) in Section 6]{Lovasz-1993},
so we omit its proof.

\begin{lemma}\label{connected}
Let $G$ be a connected Eulerian graph.
Then there exist two adjacent vertices $w$ and $u$ in $G$ such that both $G - w$ and $G-\{w,u\}$ are connected.
\end{lemma}

We also use the following lemma, which is a folklore.
For example, see \cite[Problem 6(c) in Section 6]{Lovasz-1993} for the $2$-connected case.
Notice that the solution of \cite[Problem 6(c) in Section 6]{Lovasz-1993} works 
even for a connected graph.

\begin{lemma}\label{2-connected_lem}
Let $G$ be a connected graph that is neither a cycle nor a complete graph.
Then there exist two nonadjacent vertices $x$ and $y$ in $G$
such that $G-\{x,y\}$ is connected.
\end{lemma}

\section{Proofs of Theorems} 
\label{proofs_sec}

In this section, we prove our theorems.
For a technical reason,
we allow  that an odd subgraph may have some vertices of degree 0,
while all vertices in an odd factor must have odd degree.

\subsection{Proof of Theorem~\ref{Eulerian}}
{\em Proof.} 
By Lemma~\ref{connected},
there exist two adjacent vertices $w$ and $u$ of even degree such that 
both $G - w$ and $G-\{w,u\}$ are connected.
Let $G' = G - w$. Then $G'$ has even order, $u$ has odd degree in $G'$, and $G'-u$ is connected.
By Lemma~\ref{all_edges}, $G'$ has an odd factor $F$ such that $G'-E(F)$ is a forest and
 $F$ contains all the edges in $E_{G'}(u)$.
Since $F$ contains all the edges in $E_{G'}(u)$, $w$ is the unique neighbor of $u$ in $G-E(F)$.
Moreover, since $G$ is Eulerian, for every $x \in V(G)-\{w\}$, $\deg_{G-E(F)}(x)$ is odd and $\deg_{G-E(F)-wu}(w) = \deg_G(w) -1$ is also odd.
Then $G - wu$ is decomposed into two odd subgraphs $F$ and  $G-E(F)-wu$, and hence the theorem is proved with $e = wu$.
\qed

\subsection{Proof of Theorem~\ref{thm-2}}

{\em Proof.}
Let $G$ be a $3$-connected graph of odd order.
Suppose that $G$ has at least two vertices of even degree.

If $G$ has two adjacent vertices $w$ and $u$ of even degree, then let
$P=wu$, a path consisting of one edge $wu$.
If $G$ has no two adjacent vertices of even degree, then by Theorem \ref{Tutte},
there exists a chordless path $P$
such that $P$ connects two vertices of even degree and $G - V(P)$ is connected.
We take such a path $P$ so that it is as short as possible.
Let $w$ and $u$ be the end vertices of $P$. Then both $w$ and $u$ have even degree.

Suppose that $P$ contains an internal vertex $v$ of even degree in $G$.
Let $P'$ be the subpath of $P$ from $w$ to $v$.
Since $P$ is chordless and $G$ is $3$-connected,
all vertices in $V(P) - V(P')$ has a neighbor in $G - V(P)$,
and hence $G - V(P')$ is connected.
However, this contradicts the fact that $P$ is shortest.
Therefore,
all inner vertices in $P$ have odd degree in $G$.

Let $G' = \big( G - E(P) \big)- w $. 
Since $P$ is chordless and $G$ is $3$-connected,
all vertices in $V(P) - \{w,u\}$ have a neighbor in $G - V(P)$,
and hence both $G'$ and $G' -u$ are connected.
Then $G'$ has even order,
and $u$ has odd degree in $G'$.
By Lemma~\ref{all_edges}, $G'$ has an odd factor $F$
such that $G'-E(F)$ is a forest and $F$ contains all edges in $E_{G'}(u)$.

Let $G'' = G - E(P) - E(F)$. 
Note that $u$ is an isolated vertex of $G''$.
Since $G'' - w = G'-E(F)$ is a forest,
it follows from Lemma~\ref{mfv} that
there exists a decomposition of $G''$
into $H_1, H_2$ such that $\deg_{H_i} (x)$ is odd
for every $x \in V(H_i) - \{w\}$ and each $i \in \{1,2\}$.

Let $x_0, x_1, \dots , x_t$ be the vertices in $P$
along the order from $w$ to $u$,
where $x_0 = w$ and $x_t = u$.
We initially set $E_1 = E_2 = \emptyset$,
and then recursively add the edges in $P$ to either $E_1$ or $E_2$
so that $H_i + E_i$ is an odd subgraph for $i = 1,2$.

Since $\deg_{H_1}(w) +\deg_{H_2}(w) = \deg_{G}(w) - 1$ is odd,
we may assume that 
$\deg_{H_1}(w)$ is odd and $\deg_{H_2}(w)$ is even.
Then we add the edge $x_0x_1$ to $E_2$. Note that if $P=wu$, then the proof is complete.

Suppose that the edge $x_{s-1}x_{s}$, where $1 \leq s \leq t-1$,
has been considered and added to $E_{a}$, for some $a\in \{1,2\}$.
Then we decide for the edge $x_sx_{s+1}$ as follows:
\begin{description}
\item{Case 1.}  $\deg_{G''}(x_s) \geq 1$.~~
Since $\deg_{G''}(x_s) = \deg_G(x_s) -2 - \deg_{F}(x_s)
\equiv 1 -2 -1 \equiv 0 \pmod{2}$,
both of $\deg_{H_1}(x_s)$ and $\deg_{H_2}(x_s)$ are
positive and odd.
Then, we add the edge $x_sx_{s+1}$ to $E_{a}$.
\item{Case 2.}  $\deg_{G''}(x_s) =0$.~~
We add the edge $x_sx_{s+1}$ to $E_{b}$, where $\{a,b\}=\{1,2\}$.
\end{description}

We continue this procedure until we have considered all edges in $P$,
and let $H_i' = H_i + E_i$ for $i = 1,2$.
Note that $\deg_{H_1'}(w) = \deg_{H_1}(w)$ and 
$\deg_{H_2'}(w) = \deg_{H_2}(w) + 1$,
both of which are odd.
Let $1 \leq s \leq t-1$.
In Case 1, for $a,b$ with $\{a,b\}=\{1,2\}$, 
$\deg_{H_{a}'}(x_s) = \deg_{H_{a}}(x_s) +2$
and $\deg_{H_{b}'}(x_s) = \deg_{H_{b}}(x_s)$ are both odd.
In Case 2,
$\deg_{H_1'}(x_s) = \deg_{H_2'}(x_s) =1$.
In addition,
we see that $\deg_{H_{c}'}(u) = 1=\deg_{P}(u)$ for some $c\in \{1,2\}$,
and $\deg_{H_{d}'}(u) = 0$, where $\{c,d\}=\{1,2\}$.
Therefore,
both $H_1'$ and $H_2'$ are odd subgraphs.
Consequently, $G$ is decomposed into three odd subgraphs
$F, H_1', H_2'$.
\qed

\subsection{Proof of Theorem~\ref{thm-32}}

\noindent {\em Proof.} Let $G$ be a connected graph of odd order.
Suppose that $G$ has exactly one vertex of even degree, say $w$,
such that $w$ is adjacent to all other vertices in $G$ and $G - w$ is connected.
Let $G' = G - w$, which 
is an Eulerian graph of even order.
Then $G'$ is not a complete graph.
If $G'$ is a cycle, then $G$ is a wheel.
Since every wheel of odd order has an odd $3$-edge-coloring
except for $W_4$ (see Figure~\ref{fig-1}),
we may assume that $G'$ is not a cycle.

\begin{figure}[htbp] 
\begin{center}
\includegraphics[scale=0.6]{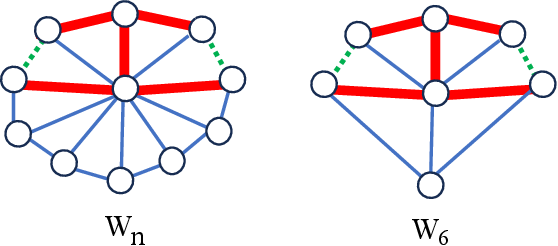}
\end{center}
\caption{An odd 3-edge-coloring of a wheel $W_n$ with even $n \ge 6$.}
\label{fig-1}  
\end{figure}

By Lemma \ref{2-connected_lem},
there exists two nonadjacent vertices $x$ and $y$ in $G'$ 
such that  $G'-\{x,y\}$ is connected.
Since $G'-\{x,y\}$ is connected and has even order,
it has an odd factor $F$.
Let $H_1$ be the subgraph of $G$ 
induced by the following edges;
all edges incident with $x$ and
all edges connecting $w$ and $V(G) - \{y\} - N_G(x)$.
Then $xw\in E(H_1)$, $\deg_{H_1}(x) = \deg_{G}(x)$ is odd,
$\deg_{H_1}(w) = |V(G)| -|\{y,w\}| -|N_G(x)-\{w\}| $ is odd,
and $\deg_{H_1}(v) = 1$ for every $v \in V(G) - \{x,y,w\}$.

Let $H_2 = G - E(F) - E(H_1)$.
Then $\deg_{H_2}(x)=0$,
$\deg_{H_2}(y) = \deg_{G}(y)$ is odd,
$\deg_{H_2}(w) = \deg_{G}(w) - \deg_{H_1}(w)$ is odd,
and 
$\deg_{H_2}(v) = \deg_{G}(v) - \deg_{F}(v) - \deg_{H_1}(v)$ is odd
for every $v \in V(G) - \{x,y,w\}$.
Thus,
$G$ is decomposed into three odd subgraphs $F, H_1, H_2$. 
\qed

\subsection{Proof of Theorem~\ref{thm-4-conn}}

\noindent {\em Proof.} 
Let $G$ be a $4$-connected graph of odd order.
By Theorems \ref{thm-2} and \ref{thm-32},
we may assume that $G$ has exactly one vertex of even degree, say $w$,
and $w$ is not adjacent to all other vertices in $G$.
We focus on proving the following;
\begin{quote}
$(*)$ $G - w$ has
two subgraphs $F_1$ and $F_2$
such that $|V(F_1)| \not\equiv |V(F_2)| \pmod{2}$,
and for each $i=1,2$,  
$N_G(w) \subseteq V(F_i)$, 
each component of $F_i$ contains a vertex in $N_{G}(w)$
and for each vertex $v$ in $F_i$,
$\deg_{F_i}(v)$ is odd if and only if $v \in N_{G}(w)$.
\end{quote}

In fact, if we have such subgraphs $F_1$ and $F_2$,
then we obtain an odd $3$-edge-coloring as follows:
Since the parities of $|V(F_1)|$ and $|V(F_2)|$ are different,
we may assume that $|V(F_1)|$ is odd.
Let $\widetilde{F}$ be the graph induced by 
$E(F_1) \cup E_G(w)$.
By the condition on $F_1$, $\widetilde{F}$ is a connected Eulerian subgraph of even order. 
Thus $\widetilde{F}$ has an odd factor $H_1$.
Since $\widetilde{F}$ is an Eulerian graph,
$\widetilde{F} - E(H_1)$ is also an odd factor of $\widetilde{F}$.
In addition, $G - E(\widetilde{F})$ is a subgraph of $G$
in which each vertex except for $w$ has an odd degree while $w$ has degree $0$.
Thus, $G$ has an odd $3$-edge-coloring such that each of $H_1$, $\widetilde{F} - E(H_1)$ and $G - E(\widetilde{F})$
forms its color class.
\medskip

Therefore, it suffices to show $(*)$.
Let $u$ be a vertex in $G$ such that $u \neq w$ and $u$ is not adjacent to $w$.
Since $G-w$ is $3$-connected,
applying Theorem~\ref{Tutte_lemma} to $G - w$ with $e$ being an edge incident to $u$
and an arbitrary vertex $v$,
there exists a chordless cycle $C$ in $(G - w) - v$ such that 
$u$ is contained in $C$ and $(G - w) - V (C)$ is connected.

For a vertex set $S$ of even number vertices of $C$,
$C$ is decomposed into two subgraphs $P_S^1$ and $P_S^2$ of $C$ such that 
for each $i = 1,2$,  $\deg_{P_S^i}(v_1)=1$ for $v_1 \in S$ and $\deg_{P_S^i}(v_2)\in \{0,2\}$ for $v_2\in V(C)-S$.
In particular, $V(P_S^1) \cap V(P_S^2) = S$ and $V(P_S^1) \cup V(P_S^2) = V(C)$.

Let $G' = (G - w) - V(C)$,
$N_{C}(w) = N_G(w) \cap V(C)$,
and $N_{G'}(w) = N_G(w) \cap V(G')$.
Then $|N_G(w)| = |N_C(w)| + |N_{G'}(w)|$, which is even, and 
$u \in V(C) - N_C(w)$.
We consider four cases according to the parities of $|V(C)|$ and $|N_C(w)|$.
Those cases employ similar ideas, but the details are different.
Thus, we need to consider them separately.

\medskip \noindent
\textbf{Case 1} ~{\em Both $|N_{C}(w)|$ and $|V(C)|$ are odd.}
\medskip 

In this case, note that $|N_{G'}(w)|$ is odd and $N_C(w) \neq \emptyset$.
Since $C$ is chordless and $G$ is $3$-connected,
there exists a vertex $x$ in $G'$ that is adjacent to $u$ in $G$.
Then,
$|N_{G'}(w) \triangle \{x\}|$ is even.
By Lemma \ref{sp_odd_sub_lemma} with $S=N_{G'}(w)\triangle\{x\}$,
$G'$ has a spanning subgraph $F'$ such that
each non-trivial component of $F'$ contains a vertex in $N_{G'}(w) \triangle \{x\}$
and for each vertex $v$ in $G'$,
$\deg_{F'}(v)$ is odd if and only if $v \in N_{G'}(w) \triangle \{x\}$.
Let $N_C' = N_C(w) \cup \{u\}$, where $|N_C'|$ is even.
For $i=1,2$,
let $F_i = F' \cup \{ux\} \cup P_{N_C'}^i$.
Since $|V(C)|$ is odd, 
$$|V(P_{N_C'}^1)| + |V(P_{N_C'}^2)| = |V(C)| + |N_C'|\equiv 1 \pmod{2}.
$$
Thus, $|V(F_1)| \not\equiv |V(F_2)| \pmod{2}$.
For $i=1,2$,
$N_G(w) \subseteq V(F_i)$ and
each non-trivial component of $F_i$ contains a vertex in $N_G(w)$
by the construction.
In addition, 
for each vertex $v$ in $F_i$,
$\deg_{F_i}(v)$ is odd if and only if 
$$v \in \left(N_{G'}(w) \triangle \{x\}\right) \triangle \{u,x\} \triangle N_C'
= N_{G'}(w) \triangle N_C(w) = N_G(w).
$$
Thus, $F_1$ and $F_2$ are the ones required in $(*)$.

\medskip \noindent
\textbf{Case 2} ~{\em $|N_{C}(w)|$ is odd and $|V(C)|$ is even.}
\medskip 

In this case, note that $|N_{G'}(w)|$ is odd and $N_C(w) \neq \emptyset$.
Let $u' \in N_C(w)$.
Since $G$ has minimum degree at least $4$ and $C$ is chordless,
there exists a vertex $x'$ adjacent to $u'$ in $G$ with $x' \neq w$.
Then,
$|N_{G'}(w) \triangle \{x'\}|$ is even.
By Lemma~\ref{sp_odd_sub_lemma} with $S=N_{G'}(w) \triangle \{x'\}$,
$G'$ has a spanning subgraph $F'$ such that
each non-trivial component of $F'$ contains a vertex in $N_{G'}(w) \triangle \{x'\}$
and
for each vertex $v$ in $G'$,
$\deg_{F'}(v)$ is odd if and only if $v \in N_{G'}(w) \triangle \{x'\}$.
Let $N_C' = N_C(w) - \{u'\}$, where $|N_C'|$ is even.
Note that $u'$ is contained in exactly one of $P_{N_C'}^1$ and $P_{N_C'}^2$,
say $u' \in V(P_{N_C'}^1)$ and $u' \notin V(P_{N_C'}^2)$ by symmetry.
For $i=1,2$,
let $F_i = (F' \cup \{u'x'\}) \cup P_{N_C'}^i$.
Note that $u' \in V(F_1) \cap V(F_2)$.
Since $|V(C)|$ is even, 
we have
$$ |V(F_1)\cap V(C)| + |V(F_2)\cap V(C)|= |V(C)| + |\{u'\}|+ |N_C'| \equiv 1 \pmod{2}. $$
Thus, $|V(F_1)| \not\equiv |V(F_2)| \pmod{2}$.

For $i=1,2$,
$N_G(w) \subseteq V(F_i)$,
each component of $F_i$ contains a vertex in $N_G(w)$,
and for each vertex $v$ in $F_i$,
$\deg_{F_i}(v)$ is odd if and only if 
$$v \in \left(N_{G'}(w) \triangle \{x'\}\right) \triangle \{u',x'\} \triangle (N_C(w) - \{u'\})
= N_{G'}(w) \triangle N_C(w) = N_G(w).
$$
Thus, $F_1$ and $F_2$ are the ones required in $(*)$.

\medskip \noindent
\textbf{Case 3} ~{\em $|N_{C}(w)|$ is even and $|V(C)|$ is odd.}
\medskip 

In this case, note that $|N_{G'}(w)|$ is even.
Thus, by Lemma \ref{sp_odd_sub_lemma} with $S = N_{G'}(w)$,
$G'$ has a spanning subgraph $F'$ such that
each non-trivial component of $F'$ contains a vertex in $N_{G'}(w)$
and
for each vertex $v$ in $G'$,
$\deg_{F'}(v)$ is odd if and only if $v \in N_{G'}(w)$.

Suppose first that $N_C(w) \neq \emptyset$.
Since $|V(C)|$ is odd, 
$|V(P_{N_C(w)}^1)| \not\equiv |V(P_{N_C(w)}^2)| \pmod{2}$. 
So, we can check that $F' \cup P_{N_C(w)}^1$ and $F' \cup P_{N_C(w)}^2$ are the ones required in $(*)$.
Thus, we may assume $N_C(w) = \emptyset$.

Since $G-w$ is $2$-connected,
there exists a path $Q$ connecting two vertices in $C$ and passing through a vertex in $F'$.
Let $F'^+$ be obtained from 
the subgraph of $G$ induced by $E(F') \triangle E(Q)$
by deleting all components that do not contain a vertex in $N_{G'}(w) =N_G(w)$.
Thus, each component of $F'^+$ contains a vertex in $N_G(w)$.
Let $U$ be the set of the end vertices of $Q$.
Note that for each vertex $v$ in $F'^+$,
$\deg_{F'^+}(v)$ is odd if and only if
$v \in N_{G}(w) \cup U$.
Since $|V(C)|$ is odd,
we have $|V(P_{U}^1)| \not\equiv |V(P_{U}^2)| \pmod{2}$. 
So, 
$F'^+ \cup P_{U}^1$ and $F'^+ \cup P_{U}^2$ are the ones required in $(*)$.

\medskip \noindent
\textbf{Case 4} ~{\em Both $|N_{C}(w)|$ and $|V(C)|$ are even.}
\medskip

In this case, note that $|N_{G'}(w)|$ is even.
Suppose first that $N_C(w) \neq \emptyset$.
Let $u' \in N_C(w)$.
Since $G$ has minimum degree at least $4$ and $C$ is chordless,
there exist two distinct vertices $x$ and $x'$ in $G'$
such that $x'$ is adjacent to $u'$ in $G$ and $x$ is adjacent to $u$.
Then,
$|N_{G'} (w)\triangle \{x,x'\}|$ is even.
By Lemma \ref{sp_odd_sub_lemma} with $S=N_{G'}(w) \triangle \{x,x'\}$,
$G'$ has a spanning subgraph $F'$ such that
each non-trivial component of $F'$ contains a vertex in $N_{G'}(w) \triangle \{x,x'\}$
and
for each vertex $v$ in $G'$,
$\deg_{F'}(v)$ is odd if and only if $v \in N_{G'}(w) \triangle \{x,x'\}$.
Let $N_C' = (N_C(w) - \{u'\}) \cup \{u\}$, where $|N_C'|$ is even.
Note that $u'$ is contained in exactly one of $P_{N_C'}^1$ and $P_{N_C'}^2$,
say $u' \in V(P_{N_C'}^1)$ and $u' \notin V(P_{N_C'}^2)$ by symmetry.
Since $|V(C)|$ is even, 
we have $|V(P_{N_C'}^1)| \equiv |V(P_{N_C'}^2)| \pmod{2}$.
For $i=1,2$,
let $F_i = (F' \cup \{ux, u'x'\}) \cup P_{N_C'}^i$.
Note that 
$$|V(F_1)| = 
|V(F')| + |V(P_{N_C'}^1)| 
\not\equiv
|V(F')| + |V(P_{N_C'}^2)| + |\{u'\}|
= |V(F_2)| \pmod{2}.
$$
For each $i =1,2$,
$N_G(w) \subseteq V(F_i)$,
each component of $F_i$ contains a vertex in $N_G(w)$,
and for each vertex $v$ in $F_i$,
$\deg_{F_i}(v)$ is odd if and only if 
\begin{align*}
v \in & \left(N_{G'}(w) \triangle \{x,x'\}\right) \triangle \{u,x,u',x'\} \triangle \left( (N_C(w) - \{u'\}) \cup \{u\}\right) \\
& = N_{G'}(w) \triangle N_(w) = N_G(w).
\end{align*}
Thus, $F_1$ and $F_2$ are the ones required in $(*)$.

Thus, we may assume $N_C(w) = \emptyset$.
Note that $N_G(w) = N_{G'}(w)$.
Since $G$ has minimum degree at least $4$ and $C$ is chordless,
there are two vertices $x_1,x_2$ in $G'$ 
such that both $x_1$ and $x_2$ are adjacent to $u$ in $G$.
Then, $|N_{G'} \triangle \{x_1,x_2\}|$ is even.
By Lemma \ref{sp_odd_sub_lemma} with $S=N_{G'}(w) \triangle \{x_1,x_2\}$,
$G'$ has a spanning subgraph $F'$ such that
each component of $F'$ contains a vertex in $N_{G'}(w) \triangle \{x_1,x_2\}$
and
for each vertex $v$ in $G'$,
$\deg_{F'}(v)$ is odd if and only if $v \in N_{G'} \triangle \{x_1,x_2\}$.
Let $F_1' = F' \cup \{ux_1, ux_2\}$ and $F_2' =F_1' \cup C$,
both of which are subgraphs of $G - w$.
Since $|V(C)|$ is even, we have 
$$|V(F_2')| = |V(F_1')| + |V(C) - \{u\}| \not\equiv |V(F_1')| \pmod{2}.$$
By the condition on $F'$,
we see that $N_G(w) \subseteq V(F_i')$,
for each vertex $v$ in $G'$ and each $i=1,2$,
$\deg_{F_i'}(v)$ is odd if and only if 
$v \in \left(N_{G'} \triangle \{x_1,x_2\}\right) \triangle \{x_1,x_2\}
= N_G(w)$.
Therefore,
if each component of $F_i'$ contains a vertex in $N_G(w)$,
then $F_1'$ and $F_2'$ are the ones required in $(*)$
and we are done.

Since each component of $F'$ contains a vertex in $N_{G'}(w) \triangle \{x_1,x_2\}$,
it follows from the handshaking lemma that 
the exceptional case occurs only when 
the component of $F'$ that contains $x_1$
also contains $x_2$ while no vertices in $N_G(w)$.
(In particular, $x_1,x_2 \notin N_{G}(w)$.)
So, we now focus on this case.

Let $F''$ be the subgraph of $G'$ obtained from $F'$ by removing 
the component containing $x_1$ and $x_2$.
Note that each component of $F''$ 
contains a vertex in $N_{G}(w)$,
and for each vertex $v$ in $G'$,
$\deg_{F''}(v)$ is odd if and only if $v \in N_{G}(w)$.

Taking a suitable path in $F'$ connecting $x_1$ and $x_2$,
together with the edges $ux_1, ux_2$,
we obtain a cycle $D'$ with $V(D') \cap N_G(w) = \emptyset$.
Note that $D'$ shares only $u$ with $C$.
If $D'$ contains no chord,
then let $D = D'$;
Otherwise,
by suitable modification, we obtain a chordless cycle $D$
such that $V(D) \subseteq V(D')$
and 
$D$ shares only $u$ with $C$.

If $|V(D)|$ is odd,
then since $D$ contains no vertex in $N_{G}(w)$,
we can apply the argument in Case 3 replacing $C$ with $D$. 
Thus, we may assume that $|V(D)|$ is even.

Let $A$ be a component of $F''$.
Since $G-w$ is $3$-connected,
there exists three pairwise vertex-disjoint paths $Q_1, Q_2, Q_3$
between $A$ and $C \cup D$.
We now assume that at least two paths,
say $Q_1$ and $Q_2$,
have end vertices in $C$,
but the other case can be shown similarly.
For $i=1,2$, let $u_i$ be the end vertex of $Q_i$ in $C$.
By combining $Q_1, Q_2$ and a path in $A$,
we obtain a path $Q$ connecting $u_1$ and $u_2$ through a vertex in $F''$.

Let $F''^+$ be obtained from the subgraph of $G$ induced by $E(F'') \triangle E(Q)$
by deleting all components that do not contain a vertex in $N_{G}(w)$.
Thus, each component of $F''^+$ contains a vertex in $N_G(w)$.
For $U = \{u_1,u_2\}$, 
we may assume that $P_{U}^1$ contains $u$. 
Let $F_1'' = F''^+ \cup P_{U}^1$,
and let $F_2'' = F_1'' \cup D$,
both of which are subgraphs of $G-w$.
Since $|V(D)|$ is even,
we have
$$|V(F_2'')| = |V(F_1'')| + |V(D) - \{u\}| \not\equiv |V(F_1'')| \pmod{2}.$$
By the condition on $F''$,
we see that $N_G(w) \subseteq V(F_i'')$,
for each vertex $v$ in $G'$ and each $i=1,2$,
$\deg_{F_i''}(v)$ is odd if and only if 
$v \in N_{G}(w)$.
Therefore,
$F_1''$ and $F_2''$ are the ones required in $(*)$.

Then, in all of the four cases,
we proved $(*)$,
which completes the proof of Theorem \ref{thm-4-conn}.
\qed
\\

\section{Toward Conjecture \ref{conj-3-conn}}
In this paper,
we show that every $4$-connected graph $G$ of odd order satisfies $\chi'_o(G) \leq 3$,
and posed Conjecture \ref{conj-3-conn}.
Toward this conjecture,
it is worthwhile to point out that 
for the case of $3$-connected graphs with minimum degree at least $4$,
the statement $(*)$ in the proof of Theorem \ref{thm-4-conn}
does not generally hold.

Let $H$ be a $3$-connected Eulerian graph of even order (e.g.~$K_6$ with perfect matching removed),
and let $uv$ be an edge in $H$.
Prepare four copies of $H$, say $H_1, H_2, H_3, H_4$,
where $u_iv_i$ is the edge in $H_i$ corresponding to $uv$ for $i=1,2,3,4$.
Let $G$ be the graph obtained from $H_i - u_iv_i$ for $i = 1,2,3,4$
and three new vertices $w, x_1, x_2$
by adding the following edges:
$$
x_1x_2, x_1u_i, x_2v_i \ (i = 1,2,3,4) ,
\text{ and } wz \text{ for any } z \in \bigcup_{i=1}^4 V(H_i).
$$
Note that $G$ is $3$-connected,
minimum degree is at least $4$,
and all vertices except for $w$ has an odd degree in $G$,
while $w$ has an even degree.

We show that $G-w$ does not contain a subgraph $F$
such that $|V(F)|$ is odd, 
$N_G(w) \subseteq V(F)$,
each component of $F$ contains a vertex in $N_G(w)$
and for each vertex $v$ in $F$,
$\deg_{F}(v)$ is odd if and only if $v \in N_{G}(w)$.
Suppose that such $F$ exists in $G-w$.
Since $|V(F)|$ is odd and $F$ contains all vertices in $N_G(w)$,
exactly one of $x_1$ and $x_2$ is contained in $F$,
say $x_1 \in V(F)$ and $x_2 \notin V(F)$.
Since $x_1$ belongs to a component containing a vertex in $N_G(w)$,
$x_1$ has to be incident to an edge in $F$,
say $x_1u_1 \in E(F)$.
Since $|V(H_1)|$ is even,
it follows from the handshaking lemma that 
there exists a vertex in $H_1$ 
that has an even degree in $F$, a contradiction.
\\

Therefore, the idea in the proof of Theorem \ref{thm-4-conn} does not work
for the case of $3$-connected graphs with minimum degree at least $4$.

Related to Conjecture \ref{conj-3-conn},
we introduce the following conjecture, 
which was asked by Yokoi \cite{Yokoi}.

\begin{conjecture}
There exists a constant $d$ satisfying the following statement:
For every $2$-connected graph $G$ of odd order with minimum degree at least $d$,
$G$ satisfies $\chi'_o(G) \leq 3$.
\label{conj-Yokoi}
\end{conjecture}

This is also open.
We point out that the similar statement for connected graphs does not hold,
since by the construction due to M\'{a}trai \cite{Matrai-2006},
we can show that 
for any constant $d$,
there exists a connected graph $G$ of odd order
with minimum degree at least $d$
such that $\chi'_o(G) = 4$.

\bigskip \noindent
{\bf Acknowledgments}
The authors thank the anonymous referee for pointing out,
in an earlier version of this manuscript,
that Theorem \ref{thm-32} holds for connected graphs.
The second author was supported by JSPS KAKENHI Grant Number JP22K13956 and JSPS KAKENHI Grant Number JP25K17301.
The third author was supported by JSPS KAKENHI Grant Numbers 22K19773, 23K03195 and 26K00616.  This work was supported by the Research Institute for Mathematical Sciences, an International Joint Usage/Research Center located in Kyoto University.

\end{document}